\documentclass{amsproc}
\usepackage{amsfonts}

\theoremstyle{thmit} 
\newtheorem{thm}{Theorem}[section]
\newtheorem{lem}[thm]{Lemma}
\newtheorem{cor}[thm]{Corollary}
\newtheorem{prop}[thm]{Proposition}

\theoremstyle{thmrm} 

\newtheorem{definition}[thm]{\bf Definition}
\newcommand{\defect}{\mathrm{def\,}}

\newcommand{\h}{\mathrm{ht}}

\newcommand{\frat}{\mathrm{Frat}}
\newcommand{\GL}{GL}
 \newcommand{\la}{\lambda}

\usepackage{graphicx,tikz}
\usepackage{float}

\title[Rank of a verbal subgroup]{On the rank of a verbal subgroup of a finite group}
\author{Eloisa Detomi}
\address{ELOISA DETOMI\\Dipartimento di Ingegneria dell'Informazione\\  Universit\`a di Padova\\ Via G. Gradenigo 6/B \\35121 Padova, Italy\\
{eloisa.detomi@unipd.it}}
\author{Marta Morigi}
\address{MARTA MORIGI\\Dipartimento di Matematica\\ Universit\`a di Bologna\\
Piazza di Porta San Donato 5 \\ 40126 Bologna, Italy\\
{marta.morigi@unibo.it}}
\author{Pavel Shumyatsky}
\address{PAVEL SHUMYATSKY\\Department of Mathematics \\University of Brasilia\\
Brasilia-DF \\ 70910-900 Brazil \\
{pavel@unb.br}}

\thanks{
The first and second authors are members of GNSAGA (Indam). The third author was partially  supported by FAPDF and CNPq.}

\begin{document}
\maketitle

\begin{abstract} We show that if $w$ is a multilinear commutator word and $G$ a finite group in which every metanilpotent subgroup generated by $w$-values is of rank at most $r$, then the rank of the verbal subgroup $w(G)$ is bounded in terms of $r$ and $w$ only. In the case where $G$ is soluble we obtain a better result -- if $G$ is a finite soluble group in which every nilpotent subgroup generated by $w$-values is of rank at most $r$, then the rank of $w(G)$ is at most $r+1$. 
\end{abstract}

\section{Introduction}
Guralnick \cite{gur} and Lucchini \cite{dp} independently proved that if all Sylow subgroups of a finite group  $G$ can be generated by $d$ elements, then the group $G$ itself can be generated by $d+1$ elements. This was an improvement over an earlier result due to Longobardi and Maj \cite{lm} giving the bound $2d$.

It follows that if all nilpotent subgroups of a finite group $G$ have rank at most $r$, then the group $G$ has rank at most $r+1$. Here, the rank of a finite group is the minimum number $r$ such that every subgroup can be generated by $r$ elements.

In the present paper we are concerned with the question whether the rank of a verbal subgroup $w(G)$ can be bounded in terms of ranks of nilpotent subgroups generated by $w$-values. Recall that, given a group $G$ and a group-word $w$, the verbal subgroup $w(G)$ is the one generated by the set $G_w$ of all $w$-values in $G$. In general, elements of $w(G)$ are not $w$-values but there are many results showing that the set $G_w$ has strong influence on the structure of $G$ (see for example \cite{Segal}). Thus, the main theme in this article is as follows.
\medskip

{\sc Question.} {\it Let $w$ be a group-word and $G$ a finite group in which every nilpotent subgroup generated by $w$-values has rank at most $r$. Is the rank of the verbal subgroup $w(G)$ bounded in terms of $r$ and $w$ only?}
\medskip

In the case where $w=x^n$ is the power word for $n\geq2$ the answer to the above question is negative. Indeed, let $G=FK$ be a Frobenius group with kernel $F$ and cyclic complement $K$ satisfying  the conditions that $F^n=1$ and $K^n=K$. It is straightforward that $G^n=G$ while any nilpotent subgroup generated by $n$th powers is cyclic. However $F$ can be chosen of arbitrarily large rank. Hence the rank of $G$ cannot be bounded in terms of $n$.

In view of this example we will focus on commutator words. Similar issues with respect to the exponent of a verbal subgroup of a finite group were addressed in \cite{S} where in particular it was proved that if $w$ is a multilinear commutator word and $G$ a finite group in which the exponent of any nilpotent subgroup generated by $w$-values divides some fixed number $e$, then the exponent of $w(G)$ is bounded in terms of $e$ and $w$ only. Later this was extended in \cite{DMS14} to some words that are not necessarily multilinear. We recall that a finite group $G$ is said to have exponent $e$ if $e$ is the least positive number such that $g^e=1$ for each $g\in G$.  
 Multilinear commutators words, also known as outer commutator words, are obtained by nesting commutators, but using always different indeterminates. For example, the word $[[x_1,x_2],[x_3,x_4,x_5],x_6]$ is a multilinear commutator word. 

The question on the rank of a verbal subgroup is more complex than that on the exponent. The main catch is that the condition that every nilpotent subgroup generated by $w$-values has rank at most $r$ may fail in homomorphic images of $G$. The corresponding condition on the exponent was shown in \cite{S} to survive under homomorphisms.

In this paper we find a way around that obstacle in the case of soluble groups. Our first result is as follows.

\begin{thm}\label{thm:sol}
Let $w$ be a multilinear commutator word and $G$ a soluble or simple  finite group  in which every nilpotent subgroup generated by $w$-values has rank at most $r$. Then the rank of the verbal subgroup $w(G)$ is at most $r+1$.
\end{thm}

Theorem \ref{thm:sol} is almost a straightforward consequence of the verification of the Ore conjecture \cite{lost} and the following proposition, which is a generalization of the Focal Subgroup Theorem to  multilinear commutator words in soluble finite groups. This might be of independent interest (see Section \ref{sec:focal}).

\begin{prop}\label{prop:focal}
Let $w$ be a multilinear commutator word and $G$ a finite soluble group. Let $P$ be a Sylow $p$-subgroup of $G$. Then $P \cap w(G)$ can be generated by $w$-values lying in $P$. 
\end{prop}

It remains unknown whether the above proposition can be extended to arbitrary finite groups. The main result of \cite{AFS} says that if $w$ is a multilinear commutator word and $G$ is a finite group, then $P\cap w(G)$ can be generated by powers of $w$-values. We also mention that  for derived words the proposition was established in \cite{dSSh}.

In the general case, that is, where $G$ is not assumed to be either soluble or simple, we were able to answer only a weaker version of the main question. Namely, the condition on the rank is imposed on the metanilpotent subgroups generated by $w$-values rather than the nilpotent ones.
\begin{thm}\label{thm:metanilp}
Let $w$ be a multilinear commutator word and $G$ a finite  group  in which every metanilpotent subgroup generated by $w$-values has rank at most $r$. Then the rank of the verbal subgroup $w(G)$ is bounded in terms of $r$ and $w$ only. 
\end{thm}

Recall that a group $G$ is metanilpotent if there is a normal subgroup $N$ such that $N$ and $G/N$ are both nilpotent. Obviously, Theorem \ref{thm:metanilp} furnishes an evidence that, for multilinear commutator words, the answer to our question should be positive. We do not know for which other words results of similar nature can be obtained. In particular, it would be interesting to see if Theorems \ref{thm:sol} and \ref{thm:metanilp} remain valid with $w$ being an Engel word.

\section{Multilinear commutator words}\label{sec:mcw}

We  recall that multilinear commutator words are recursively defined as follows:  the group-word $w(x)=x$ in one indeterminate is a multilinear commutator; if    $\alpha$ and $\beta$ are  multilinear commutators involving disjoint sets of indeterminates, then the word $w=[\alpha, \beta]$  
is a multilinear commutator, and all multilinear commutators are obtained in this way. Examples of multilinear commutators include the familiar lower central words $\gamma_n(x_1,\dots,x_n)=[x_1,\dots,x_n]$ and derived words $\delta_n$, on $2^n$ variables, defined recursively by
$$\delta_0=x_1,\qquad \delta_n=[\delta_{n-1}(x_1,\ldots,x_{2^{n-1}}),\delta_{n-1}(x_{2^{n-1}+1},\ldots,x_{2^n})].$$
Of course $\delta_n(G)=G^{(n)}$ is the $n$th derived group of a group $G$.

Let $X$ be a subset of a group and $w=w(x_1,\dots,x_n)$ a multilinear commutator word. We set
$$X_w=\{w(y_1,\dots,y_n) \mid y_i\in X\}.$$ Say that the subset $X$ is commutator-closed if $[x,y]\in X$ whenever $x,y\in X$.

If $G=\langle X \rangle$ is a group generated by a commutator-closed set $X$, then the commutator subgroup $G'$ is generated by commutators $[x,y]$, where $x,y \in X$ (see  \cite[Lemma 2.2]{dSSh}). Here we will generalize this result to multilinear commutators words. 

\begin{prop}\label{goodgen}
Let $G=\langle X \rangle$ be a group generated by a commu\-tator-closed set $X$ and let $w$
be a multilinear commutator word. Then $w(G)=\langle X_w \rangle$. 
\end{prop}

To prove Proposition \ref{goodgen} we will need the concepts of {\it height} and {\it defect} of a multilinear commutator word, introduced in \cite{FM}. For the reader's convenience, we will now describe some results from \cite{FM}.

\begin{definition}
The height and the labelled tree of a multilinear commutator word 
are defined recursively as follows:
\begin{enumerate}
\item
An indeterminate has height $0$, and its tree is an isolated vertex,
labelled with the name of the indeterminate.
\item
If $w=[\alpha,\beta]$, where $\alpha$ and $\beta$ are disjoint multilinear commutator words,  
 then 
 the height $\h(w)$ of the word $w$ is taken to be the maximum of the heights of $\alpha$
and $\beta$ plus $1$, and the tree of $w$ is obtained by adding a new vertex with label $w$
and connecting it to the vertices labelled $\alpha$ and $\beta$ of the corresponding trees of these words.
\end{enumerate}
\end{definition}

The tree of a multilinear commutator word $w$ provides a visual way of reading how $w$
is constructed by nesting commutators, easier than writing the actual expression of $w$ by
using commutator brackets.
We draw these trees by going downwards whenever we form a new commutator, so that the vertex
with label $w$ is placed at the root of the tree.
Every vertex $v$ is labelled with a multilinear commutator word, which we denote by $w_v$.
Note that the indeterminates correspond exactly to the vertices of degree $1$.
Also, the height of $w$ coincides with the height of the tree, that is, the largest distance
from the root to another vertex of the tree (which will be necessarily labelled by an indeterminate). For example, the following are the trees for the words $\gamma_4$ and $\delta_3$:
\begin{figure}[H]
\centering
\begin{tikzpicture}[level distance=5mm]
\tikzstyle{level 1}=[sibling distance=10mm]
\tikzstyle{level 2}=[sibling distance=10mm]
\tikzstyle{level 3}=[sibling distance=10mm]
\coordinate (root)[grow=up, fill] circle (2pt)
child {[fill] circle (2pt)}
child {[fill] circle (2pt)
       child {[fill] circle (2pt)}
       child {[fill] circle (2pt)
               child {[fill] circle (2pt)}
               child {[fill] circle (2pt)}
             }
      };
\node[below=2pt] at (root) {\scriptsize $\gamma_4$};
\node[left=4pt] at (root-2) {\scriptsize $[x_1,x_2,x_3]$};
\node[right=2pt] at (root-1) {\scriptsize $x_4$};
\node[left=2pt] at (root-2-2) {\scriptsize $[x_1,x_2]$};
\node[right=2pt] at (root-2-1) {\scriptsize $x_3$};
\node[above=2pt] at (root-2-2-2) {\scriptsize $x_1$};
\node[above=2pt] at (root-2-2-1) {\scriptsize $x_2$};
\end{tikzpicture}
\qquad
\begin{tikzpicture}[level distance=5mm]
\tikzstyle{level 1}=[sibling distance=26mm]
\tikzstyle{level 2}=[sibling distance=15mm]
\tikzstyle{level 3}=[sibling distance=5mm]
\coordinate (root)[grow=up, fill] circle (2pt)
child {[fill] circle (2pt)
       child {[fill] circle (2pt)
               child {[fill] circle (2pt)}
               child {[fill] circle (2pt)}
             }
       child {[fill] circle (2pt)
               child {[fill] circle (2pt)}
               child {[fill] circle (2pt)}
             }
      }
child {[fill] circle (2pt)
       child {[fill] circle (2pt)
               child {[fill] circle (2pt)}
               child {[fill] circle (2pt)}
             }
       child {[fill] circle (2pt)
               child {[fill] circle (2pt)}
               child {[fill] circle (2pt)}
             }
      };
\node[below=2pt] at (root) {\scriptsize $\delta_3$};
\node[left=4pt] at (root-2) {\scriptsize $[[x_1,x_2],[x_3,x_4]]$};
\node[right=3pt] at (root-1) {\scriptsize $[[x_5,x_6],[x_7,x_8]]$};
\node[left=2pt] at (root-2-2) {\scriptsize $[x_1,x_2]$};
\node[left=3pt] at (root-2-1) {\scriptsize $[x_3,x_4]$};
\node[right=3pt] at (root-1-2) {\scriptsize $[x_5,x_6]$};
\node[right=2pt] at (root-1-1) {\scriptsize $[x_7,x_8]$};
\node[above=2pt] at (root-2-2-2) {\scriptsize $x_1$};
\node[above=2pt] at (root-2-2-1) {\scriptsize $x_2$};
\node[above=2pt] at (root-2-1-2) {\scriptsize $x_3$};
\node[above=2pt] at (root-2-1-1) {\scriptsize $x_4$};
\node[above=2pt] at (root-1-2-2) {\scriptsize $x_5$};
\node[above=2pt] at (root-1-2-1) {\scriptsize $x_6$};
\node[above=2pt] at (root-1-1-2) {\scriptsize $x_7$};
\node[above=2pt] at (root-1-1-1) {\scriptsize $x_8$};
\end{tikzpicture}
\caption{The trees of the words $\gamma_4$ and $\delta_3$.}
\end{figure}
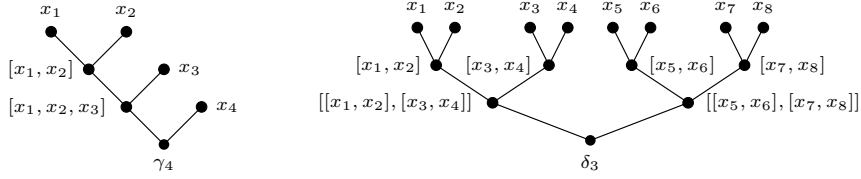
More generally, the full tree of height $h$ corresponds to the derived word $\delta_h$.

All labels of the tree of a multilinear commutator word are  determined,
up to permuting the indeterminates, by the tree itself (as a graph without labels): given the tree,
we only need to associate an indeterminate to every vertex of degree $1$, and then proceed
downwards by labeling each vertex with the commutator of the labels of its immediate
ascendants.

\begin{definition}
Let $v$ be a vertex of the tree of a multilinear commutator word $w$ of height $h$.
We say that $v$ is in the {\em $i$-th level} of the tree if it lies at distance $h-i$
from the root of the tree.
\end{definition}

Thus the upmost level will be level $0$ and the root will be at level $h$, but note
that a vertex $v$ at level $i$ is not necessarily labelled with a word $w_v$ of
height $i$, it might even happen that $w_v$ is an indeterminate.

It is also useful to associate a {\em companion} vertex to each vertex of the tree
different from the root, defined as follows.

\begin{definition}
Let $p$ be a vertex of the tree of a multilinear commutator word $w$, different from
the root, and let $u$ be the immediate descendant of $p$.
Then the {\em companion} of $p$ is the only other vertex $q$ of the tree which has $u$
as an immediate descendant.
\end{definition}

It is clear that companion vertices lie on the same level of the tree.

\medskip

We will prove Proposition \ref{goodgen} for a general multilinear commutator word
$w$ by induction on the `distance' of $w$ to the closest derived word.
We make this notion of distance precise in the following definition.

\begin{definition}
Let $w$ be a multilinear commutator word of height $h$.
Then the {\em defect\/} of $w$, which is denoted by $\defect w$, is defined as
\[
\defect w= 2^{h+1} - 1 - V,
\]
where $V$ is the number of vertices of the tree of $w$.
\end{definition}

So, if the height of $w$ is $h$, then the defect is the number of vertices that need to be
added to the tree of $w$ in order to get the tree of $\delta_h$.
Thus the defect is $0$ if and only if $w$ is a derived word, and we have $\defect\gamma_4=8$
and $\defect[\gamma_3,\gamma_3]=4$.

\begin{definition}
Let $T$ be the tree associated to a multilinear commutator word $w$.
A subset $S$ of vertices of $T$ is called a {\em section\/} of $T$ if $S$ is maximal (with
respect to inclusion) subject to the condition that $S$ does not contain two vertices which
are one a descendant of the other.
Equivalently, in terms of labels, this means that every indeterminate involved in $w$
appears in exactly one word $w_v$ with $v\in S$.
\end{definition}

A very natural way of obtaining a section is by cutting a tree below level $i$, that is,
we consider the section $S$ containing all vertices at level $i+1$ and all the vertices
of the tree lying below level $i+1$ labelled by an indeterminate.
This is the type of section that we will use in the proof of Proposition \ref{goodgen}.
\begin{figure}[H]
\centering
\begin{tikzpicture}[level distance=5mm]
\tikzstyle{level 1}=[sibling distance=26mm]
\tikzstyle{level 2}=[sibling distance=13mm]
\tikzstyle{level 3}=[sibling distance=7mm]
\tikzstyle{level 4}=[sibling distance=4mm]
\coordinate (root)[grow=up, fill] circle (2pt)
child {[fill] circle (2pt)
       child {[fill] circle (2pt)}
       child {[fill] circle (2pt)
               child {[fill] circle (2pt)}
               child {[fill] circle (2pt)
                      child {[fill] circle (2pt)}
                      child {[fill] circle (2pt)}
                     }
             }
      }
child {[fill] circle (2pt)
       child {[fill] circle (2pt)}
       child {[fill] circle (2pt)
               child {[fill] circle (2pt)}
               child {[fill] circle (2pt)
                      child {[fill] circle (2pt)}
                      child {[fill] circle (2pt)}
                     }
             }
      };
\node[below=2pt] at (root) {\scriptsize $[\gamma_4,\gamma_4]$};
\node[transparent] (a) at (root-2-2-2) {\scriptsize aa}
node[transparent] (b) at (root-2-2-1) {\scriptsize a}
node[transparent] (c) at (root-2-1) {\scriptsize a}
node[transparent] (d) at (root-1-2-2) {\scriptsize a}
node[transparent] (e) at (root-1-2-1) {\scriptsize a}
node[transparent] (f) at (root-1-1) {\scriptsize aa};
\draw[dashed,rounded corners] (a.west) -- (a.north) -- (b.north) -- (c.north)
-- (d.north) -- (e.north) -- (f.north) -- (f.east) -- (f.south) -- (e.south)
-- (d.south) -- (c.south) -- (b.south) -- (a.south) -- cycle;
\node[right=3mm] at (root-1-1) {\scriptsize $S$};
\end{tikzpicture}
\end{figure}

\begin{lem}
\label{generalized 3 subgroup lemma}\cite[Lemma 3,2]{FM}
Let $w$ be a multilinear commutator word, and let $T$ be the tree of $w$.
If $\eta$ is another multilinear commutator word, then for every $v\in T$, we define $\pi^{(v)}$ to be
the word whose tree is obtained by replacing the tree of $w_v$ at vertex $v$ with the
tree of $[w_v,\eta]$.
Then, for every section $S$ of $T$, and for every group $G$, we have
\[
[w(G),\eta(G)]
\le
\prod_{v\in S} \, \pi^{(v)}(G).
\]
\end{lem}

\begin{definition}
Let $\varphi$ and $ w $ be two multilinear commutator words.
Then:
\begin{enumerate}
\item
We say that $ w $ is a {\em constituent\/} of $\varphi$ if $ w $ is
the label of a vertex in the tree of $\varphi$.
\item
We say that $\varphi$ is an {\em extension\/} of $ w $ if the tree of $\varphi$ is an upward extension of the tree
of $ w $ (simply as a tree, without labels).
\end{enumerate}
\end{definition}

Thus, in order to get an extension of $ w $, we only need to draw new binary trees at some of the
vertices which are labelled by indeterminates in the tree of $ w $.
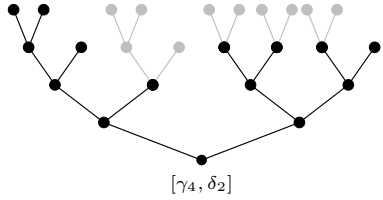
\begin{figure}[H]
\centering
\begin{tikzpicture}[level distance=5mm]
\tikzstyle{level 1}=[sibling distance=26mm]
\tikzstyle{level 2}=[sibling distance=13mm]
\tikzstyle{level 3}=[sibling distance=7mm]
\tikzstyle{level 4}=[sibling distance=4mm]
\coordinate (root)[grow=up, fill] circle (2pt)
child {[fill] circle (2pt)
       child {[fill] circle (2pt)
               child {[fill] circle (2pt)}
               child {[fill] circle (2pt)
                      child[gray!50] {[fill] circle (2pt)}
                      child[gray!50] {[fill] circle (2pt)}
                     }
             }
       child {[fill] circle (2pt)
               child {[fill] circle (2pt)
                      child[gray!50] {[fill] circle (2pt)}
                      child[gray!50] {[fill] circle (2pt)}
                     }
               child {[fill] circle (2pt)
                      child[gray!50] {[fill] circle (2pt)}
                      child[gray!50] {[fill] circle (2pt)}
                     }
             }
      }
child {[fill] circle (2pt)
       child {[fill] circle (2pt)
               child[gray!50] {[fill] circle (2pt)}
               child[gray!50] {[fill] circle (2pt)
                      child {[fill] circle (2pt)}
                      child {[fill] circle (2pt)}
                     }
             }
       child {[fill] circle (2pt)
               child {[fill] circle (2pt)}
               child {[fill] circle (2pt)
                      child {[fill] circle (2pt)}
                      child {[fill] circle (2pt)}
                     }
             }
      };
\node[below=2pt] at (root) {\scriptsize $[\gamma_4,\delta_2]$};
\draw[fill] (root-2-1) circle (2pt);
\draw[fill] (root-1-2-1) circle (2pt);
\draw[fill] (root-1-2-2) circle (2pt);
\draw[fill] (root-1-1-2) circle (2pt);
\end{tikzpicture}\label{extensionfigure}
\end{figure}

In Figure~\ref{extensionfigure}, the black tree represents the word $ w =[\gamma_4,\delta_2]$,
and the extension of $ w $ which is obtained by adding the grey trees is
$\varphi=[[\gamma_3,\gamma_3],[\delta_2,\gamma_3]]$.
Clearly,
 the derived word $\delta_h$ is an extension of all words of height
less than or equal to $h$.

Observe that, if $G$ is a group and  $\varphi$ is an extension of $ w $, then $\varphi (G) \le w(G)$. Moreover, if $\alpha$ is a constituent of $w$,  
 then $w(G) \le \alpha (G)$.

The proof of Proposition \ref{goodgen} depends on the following result which is implicit in the proof of Theorem B of \cite{FM}. 

\begin{thm} \label{thm:implicit} 
Let $w=[\alpha,\beta]$ be a multilinear commutator word of height $h$. If $w \neq \delta_h$, then at least one of the subgroups $[w(G),\alpha(G)]$ and $[w(G),\beta(G)]$ is contained in a product of verbal subgroups corresponding to words which are proper extensions of $w$ of height $h$. 
\end{thm}

For the reader's convenience we include a proof.

\begin{proof}
Let $\Phi$ be the (finite) set of all multilinear commutator words of height $h$ which are a proper
extension of $ w $ and set $H=\prod_{\varphi \in \Phi}\varphi(G)$.

Let $i$ be the largest integer for which there is a vertex in the tree of $ w $ at level $i$
with label $\delta_i$.
Note that $1\le i<h$, since $ w $ is not a derived word.
Let $S$ be the  section of the tree of $ w $ obtained by cutting the tree below level $i$, so
that $S$ contains all vertices at level $i+1$ and all the vertices of the tree lying below level
$i+1$ which are labelled with an indeterminate.
For every vertex $v$ in $S$, we construct a word $ w ^{(v)}$ as follows.
If the label $ w _v$ of $v$ is not an indeterminate, then we can write $ w _v=[ w _p, w _q]$,
where $p$ and $q$ are the companion vertices at level $i$ having $v$ as immediate descendant.
By the maximality of $i$, one of these vertices is labelled with a word which is different from $\delta_i$.
For simplicity, let us assume that this happens for $q$, the vertex on the right (the argument is exactly
the same otherwise).
We define $ w ^{(v)}$ to be the word whose tree is obtained by replacing $ w _q$ with $\delta_i$ in
the tree of $ w $.
Thus the label of $ w ^{(v)}$ at the vertex $v$ is the commutator $[ w _p,\delta_i]$.
On the other hand, if $ w _v$ is an indeterminate, then $ w ^{(v)}$ is defined simply by
putting the tree corresponding to $\delta_i$ on top of the vertex $v$ in the tree of $ w $.

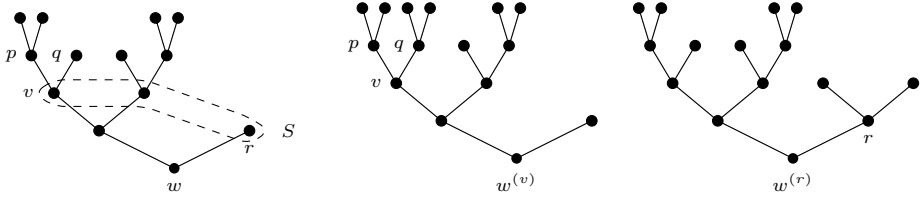
\begin{figure}[H]
\centering
\begin{tikzpicture}[level distance=5mm]
\tikzstyle{level 1}=[sibling distance=20mm]
\tikzstyle{level 2}=[sibling distance=12mm]
\tikzstyle{level 3}=[sibling distance=6mm]
\tikzstyle{level 4}=[sibling distance=3mm]
\coordinate (root)[grow=up, fill] circle (2pt)
child {[fill] circle (2pt)  }
child {[fill] circle (2pt)
       child {[fill] circle (2pt)
               child {[fill] circle (2pt)
                      child {[fill] circle (2pt)}
                      child {[fill] circle (2pt)}
                     }
              child {[fill] circle (2pt)}
             }
       child {[fill] circle (2pt)
               child {[fill] circle (2pt)}
               child {[fill] circle (2pt)
                      child {[fill] circle (2pt)}
                      child {[fill] circle (2pt)}
                     }
             }
      };
\node[transparent] (a) at (root-2-2) {\scriptsize aa}
node[transparent] (b) at (root-2-1) {\scriptsize a}
node[transparent] (c) at (root-1){\scriptsize aa};
\draw[dashed,rounded corners] (a.west) -- (a.north) -- (b.north)  --
(c.north) -- (c.east) -- (c.south) -- (b.south) -- (a.south) -- cycle;
\node[right=3mm] at (root-1) {\scriptsize $S$};
\node[left=2pt] at (root-2-2-2) {\scriptsize $p$};
\node[left=2pt] at (root-2-2-1) {\scriptsize $q$};
\node[left=4pt] at (root-2-2) {\scriptsize $v$};
\node[below=2pt] at (root-1) {\scriptsize $r$};
\node[below=2pt] at (root) {\scriptsize $ w $};
\end{tikzpicture}
\quad
\begin{tikzpicture}[level distance=5mm]
\tikzstyle{level 1}=[sibling distance=20mm]
\tikzstyle{level 2}=[sibling distance=12mm]
\tikzstyle{level 3}=[sibling distance=6mm]
\tikzstyle{level 4}=[sibling distance=3mm]
\coordinate (root)[grow=up, fill] circle (2pt)
child {[fill] circle (2pt)  }
child {[fill] circle (2pt)
       child {[fill] circle (2pt)
               child {[fill] circle (2pt)
                      child {[fill] circle (2pt)}
                      child {[fill] circle (2pt)}
                     }
              child {[fill] circle (2pt)}
             }
       child {[fill] circle (2pt)
               child {[fill] circle (2pt)
               child {[fill] circle (2pt)}
                      child {[fill] circle (2pt)}
                      }
               child {[fill] circle (2pt)
                      child {[fill] circle (2pt)}
                      child {[fill] circle (2pt)}
                     }
             }
      };
\node[left=2pt] at (root-2-2-2) {\scriptsize $p$};
\node[left=2pt] at (root-2-2-1) {\scriptsize $q$};
\node[left=2pt] at (root-2-2) {\scriptsize $v$};
\node[below=2pt] at (root) {\scriptsize $ w ^{(v)}$};
\end{tikzpicture}
\quad
\begin{tikzpicture}[level distance=5mm]
\tikzstyle{level 1}=[sibling distance=20mm]
\tikzstyle{level 2}=[sibling distance=12mm]
\tikzstyle{level 3}=[sibling distance=6mm]
\tikzstyle{level 4}=[sibling distance=3mm]
\coordinate (root)[grow=up, fill] circle (2pt)
child {[fill] circle (2pt)
                      child {[fill] circle (2pt)}
                      child {[fill] circle (2pt)} }
                      child {[fill] circle (2pt)
       child {[fill] circle (2pt)
               child {[fill] circle (2pt)
                      child {[fill] circle (2pt)}
                      child {[fill] circle (2pt)}
                     }
              child {[fill] circle (2pt)}
             }
       child {[fill] circle (2pt)
               child {[fill] circle (2pt)}
               child {[fill] circle (2pt)
                      child {[fill] circle (2pt)}
                      child {[fill] circle (2pt)}
                     }
             }
      };
\node[below=2pt] at (root-1) {\scriptsize $r$};
\node[below=2pt] at (root) {\scriptsize $ w ^{(r)}$};
\end{tikzpicture}
\caption{\label{section}The two different cases for the construction of $ w ^{(v)}$ with $v\in S$.
Observe that $i=1$ in this example.}
\end{figure}

In any case, it is clear that $\h( w ^{(v)})=h$ and that $ w ^{(v)}$ is a proper extension of
 $ w $, so that $ w ^{(v)}$ belongs to $\Phi$. Consequently, we have $ w ^{(v)}(G)\le H$ for every vertex $v$ in the section $S$.

On the other hand, if we apply Lemma \ref{generalized 3 subgroup lemma} to the section $S$ with
$\delta_i$ playing the role of $\eta$, then we have
\begin{equation}
\label{omega-delta}
[ w (G),\delta_i(G)]
\le
\prod_{v\in S} \, \pi^{(v)}(G).
\end{equation}
Here, $\pi^{(v)}$ is the word whose tree is obtained by inserting the tree of $[ w _v,\delta_i]$
at vertex $v$ in the tree of $ w $.
Now, it is easy to compare the two words $ w ^{(v)}$ and $\pi^{(v)}$: they look the same at all
vertices of the original tree of $ w $, except for the vertex $v$, where $\pi^{(v)}$ has the label
$[ w _v,\delta_i]$ and $ w ^{(v)}$ has either $[ w _p,\delta_i]$ or $\delta_i$.
In any of the two cases, we have
\[
(\pi^{(v)})_v(G) \le ( w ^{(v)})_v(G),
\]
and then, since $\pi^{(v)}$ and $ w ^{(v)}$ have the same labels outside the tree above $v$,
also
\[
\pi^{(v)}(G) \le  w ^{(v)}(G).
\]
Since this happens for all vertices in $S$, it follows from (\ref{omega-delta}) that
$$[ w (G),\delta_i(G)]\le H.$$
Now, by the definition of $i$, the derived word $\delta_i$ is a constituent of either $\alpha$ or
$\beta$, and consequently either $[ w (G),\alpha(G)]\le H$ or $[ w (G),\beta(G)]\le H$.
This concludes the proof of the theorem.
\end{proof} 

We will also need the following observation. 
 
\begin{lem}\label{T} Let $H,K$ be two subgroups of a group $G$ such that $H=\langle A \rangle$, $K=\langle B \rangle$ and let $C=\{[a,b]\,|\, a\in A,\,b\in B\}$.
Let $T=\langle C\rangle$. If $[a,b]^{\tilde a},[a,b]^{\tilde b}\in T$ for all $a,\tilde a\in A$ and $b,\tilde b\in B$, then $T=[H,K]$.
\end{lem}
\begin{proof}
It follows from the hypotheses that $T$ is a normal subgroup of $\langle H,K\rangle$, thus in the quotient group $\langle \bar H,\bar K\rangle=\langle HT/T,KT/T\rangle$ we have that every generator of $\bar H$ commutes with every
generator of $\bar K$. Therefore  
 $[\bar H,\bar K]=1$, that is $[H,K]\le T$. The reverse inclusion is obvious.
\end{proof}

\begin{proof}[{\textsc{Proof of Proposition \ref{goodgen}}}]
We argue by double induction: we first use induction on the height of the word $ w $, and then, for a fixed value of the height, induction on the defect of $ w $.
If $ w $ has height $0$, then $ w =x_1$ and the result holds.
Now assume that $h=\h( w )\ge 1 $ and that the result has been proved for any multilinear commutator word whose height is less than $h$. If $\defect w =0$, then $ w $ is a derived word, and the result holds by  \cite[Lemma 2.2]{dSSh}. So we assume that $\defect w >0$.
Let us write $ w =[\alpha,\beta]$, where $\alpha$ and $\beta$ are multilinear commutator words of height smaller than $h$. Then, by induction on the height, $\alpha(G)=\langle X_\alpha\rangle$ and $\beta(G)=\langle X_\beta\rangle$.

Let $\Phi$ be the (finite) set of all multilinear commutator words of height $h$ which are a proper
extension of $ w $. Then, by induction on the defect,  $\varphi(G)=\langle X_\varphi\rangle$ for each $\varphi \in \Phi$. 

Let  $T=\langle X_w\rangle$. Since $X$ is commutator-closed, we have that $X_\varphi\subseteq X_w$ for each $\varphi \in \Phi$. Therefore $$\prod_{\varphi \in \Phi}\varphi(G)\le T.$$
In view of Theorem \ref{thm:implicit}, at least one of the subgroups $[ w (G),\alpha(G)]$ and
$[ w (G),\beta(G)]$ is contained  in $T$. 

Now, by Lemma \ref{T}, in order to prove that $T=w(G)$ it is enough to prove that 
$[a,b]^{\tilde a},[a,b]^{\tilde b}\in T$ for all $a,\tilde a\in X_\alpha$ and $b,\tilde b\in X_\beta$.

Let us assume that $[ w (G),\alpha(G)]\le T$, the other case being similar.

Then $[a,b]^{\tilde a}=[a,b][a,b,{\tilde a}]\in T$, as $[a,b]\in X_w$ and $[a,b,{\tilde a}]\in [ w (G),\alpha(G)]\le T$. This proves that $T$ is normalized by $\alpha(G)$.

Moreover, as $[a,b]\in \alpha(G)=\langle X_{\alpha}\rangle$, we can write $[a,b]=c_1c_2\cdots c_r$ with $c_i\in X_{\alpha}\cup{X_{\alpha}}^{-1}$. Then $[a,b]^{\tilde b}=[a,b][a,b,{\tilde b}]$, and by the standard commutator identities we can write $[a,b,{\tilde b}]=[c_1c_2\cdots c_r,\tilde b]$ as the
product of $r$ ${\alpha}(G)$-conjugates of elements of the form $[c_i,\tilde b]$.  
 If $c_i\in  X_\alpha$, then $[c_i,\tilde b]\in X_w\le T$. If $c_i^{-1}\in X_\alpha$,
then again $[c_i,\tilde b]=
([c_i^{-1},\tilde b]^{-1})^{c_i}\in T^{\alpha(G)}=T$.
It follows that $[a,b,{\tilde b}]\in T$ and thus $[a,b]^{\tilde b}=[a,b][a,b,{\tilde b}]\in T$, as desired.
\end{proof}

\section{Proof of Proposition \ref{prop:focal}}\label{sec:focal}

The Focal Subgroup Theorem (see e.g. \cite[Theorem 7.3.4]{gor}) says that if $P$ is a
Sylow subgroup of a finite group $G$, then $P\cap G'$ is generated by elements
of the form $[x,y]\in P$, where $x\in P$ and $y\in G$. In particular, it follows
that the Sylow subgroups of $G'$ are generated by commutators. Thus, the
following question arises.

{\emph{ Let $w$ be a commutator word, $G$ a finite group and $P$ a Sylow $p$-subgroup
of $w(G).$ Is it true that $P$ can be generated by $w$-values lying in $P$?}}

The above question was introduced in \cite{AFS} where it was proved that if $w$ is a multilinear commutator word, then $P$ is generated by powers of $w$-values. In this section we prove that if $G$ is soluble, then indeed $P$ can be generated by $w$-values. For the derived words this was established in \cite[Lemma 2.5]{dSSh}. To deal with arbitrary multilinear commutator words, we require the following combinatorial lemma.

Let $i$ be a positive integer. We denote by $I$ the set of all $n$-tuples $(i_1,\dots,i_n)$, where all entries $i_k$ are non-negative integers. We will view $I$ as a partially ordered set with the partial order given by the rule that $$(i_1,\dots,i_n)\leq(j_1,\dots,j_n)$$ if and only if $i_1\leq j_1,\dots,i_n\leq j_n$.

Given a group $G$, a  multilinear commutator word  $w=w(x_1,\dots,x_n)$ and  
 $\mathbf{i}=(i_1,\ldots,i_n) \in I$, we write
\[
w(\mathbf{i})=w(G^{(i_1)},\ldots,G^{(i_n)})
\]
 for  the subgroup generated by the $w$-values $w(g_1,\dots,g_n)$ 
 with  $g_{j} \in G^{(i_j)}$. 
 Further, set 
\[
w(\mathbf{i^+})=\prod w(\mathbf{j} ),
\]
 where the product is taken over all $\mathbf{j} \in I $ such that $\mathbf{j}>\mathbf{i}$.

 Observe that $ w(\mathbf{i}) =  w_{\mathbf{i}}(G)$ where $  w_{\mathbf{i}}$ is the extension of $w$ obtained by replacing,   in $w(x_1,\dots,x_n)$, 
each $x_j$   with the word $\delta_{i_j}$, for $j=1, \dots, n$.

Note that if $\delta_h(G)=1$, then there is an $n$-tuple $\mathbf{i}$ such that $w(\mathbf{i}) \neq 1$ but $w(\mathbf{i^+})=1$.

\begin{lem}\cite[Corollary 6]{DMS-coverings} \label{cor:ab}
 Let $G$ be a group, $w=w(x_1,\ldots,x_n)$  a  multilinear commutator word and $\mathbf{i} \in I$. If $w(\mathbf{i}^+)=1$, then $w(\mathbf{i})$ 
 is abelian.  
\end{lem}

The next lemma is taken from \cite{AFS}.

\begin{lem}\cite[Lemma 1.1]{AFS} \label{p-set}
Let $G$ be a finite group  and $N$ be a normal subgroup of $G$. If $P$ is a Sylow $p$-subgroup of $G$ and $Y$ is a normal subset of $G$ consisting of $p$-elements, then $YN \cap PN=(Y \cap P) N$. 
\end{lem}

The following lemma  will also play an important role.
\begin{lem}\cite[Lemma 2.1]{dSSh}\label{sol-gen}
Any finite soluble group is generated by a commu\-ta\-tor-closed set all of whose elements have prime power order.
\end{lem}

Now we are ready to prove that if  $w$ is a multilinear commutator word and $G$ is a finite soluble group, then for any Sylow $p$-subgroup $P$ of $G$ the corresponding Sylow $p$-subgroup $P \cap w(G)$ of $w(G)$ can be generated 
by the $w$-values lying in $P$.

\begin{proof}[{\textsc{Proof of Proposition \ref{prop:focal}}}] Recall that $G$ is a finite soluble group and $w$ is a multilinear commutator word. We know from Lemma \ref{sol-gen} that there is a commutator-closed subset $X\subseteq G$ such that $X$ generates $G$ and every element of $X$ has prime power order. 

Recall that $X_w$ stands for the set $\{w(x_1, \dots , x_n) \mid x_i \in X \}$. For a prime $p\in\pi(G)$ set 
$$X_{w,p}= \{ x \in X_w \mid x \ \textrm{is a $p$-element} \}$$ and 
$$Y_{w,p}=\{ X_{w,p}^G \},$$ that is, $Y_{w,p}$ is the union of the conjugacy classes of elements of $X_{w,p}$. 

We claim that if $P$ is a Sylow $p$-subgroup of $G$, then $P\cap w(G)=\langle P\cap Y_{w,p}\rangle$. 

Without loss of generality assume that $G$ is a minimal counterexample to the above claim. 
 If $N$ is a nontrivial normal subgroup of $G$, then the set $\bar X= \{ xN \mid x \in X \}$ has the required properties.
 Since $X$ is a commutator-closed set, we deduce that every element of $X_w$ has prime power order. Hence, 
 $$(\bar X)_{w,p} =\overline{ X_{w,p}}=\{  xN \in \bar X_w \mid xN \ \textrm{is a $p$-element } \},$$ so that $\bar Y_{w,p} =\overline{Y_{w,p}}$. By minimality of $G$, 
$$\bar P \cap w(\bar G)=\langle \bar P \cap \bar Y_{w,p} \rangle = \langle \bar P \cap \overline{ Y_{w,p}} \rangle.$$ 
By virtue of Lemma \ref{p-set},
\begin{equation}\label{N}
P \cap w(G)= \langle P \cap N, P \cap Y_{w,p} \rangle.
\end{equation}
So if $N$ is a $p'$-group, then $P \cap N=1 $ and $P \cap w(G)= \langle  P \cap Y_{w,p} \rangle$, a contradiction. Hence $G$ has no nontrivial normal $p'$-subgroups. 

We will now use the notation introduced before Lemma \ref{cor:ab}.
Let $\mathbf{i} \in I$ such that $w(\mathbf{i}) \neq 1$ but $w(\mathbf{i^+})=1$. By Lemma \ref{cor:ab}, $ w(\mathbf{i})$ 
 is abelian.  Since $G$ has no nontrivial normal $p'$-subgroups, $w(\mathbf{i})$ is a $p$-group. 
It follows from Proposition \ref{goodgen} that $$ w(\mathbf{i}) =  w_{\mathbf{i}}(G)= \langle X_{w_{\mathbf{i}}} \rangle.$$
 Since $ w(\mathbf{i})$ is a $p$-group, $X_{w_{\mathbf{i}}}= X_{w_{\mathbf{i},p}}$ and $P \cap   w(\mathbf{i})= w(\mathbf{i})$. 
 Moreover, every $w_{\mathbf{i}}$-value is a $w$-value, whence $X_{w_{\mathbf{i},p}} \le Y_{w,p}$. 
 We deduce from  (\ref{N}), applied with $N= w(\mathbf{i})$, that 
  $$P \cap w(G)= \langle  P \cap Y_{w,p}, P \cap w(\mathbf{i}) \rangle
 =  \langle  P \cap Y_{w,p},  w(\mathbf{i}) \rangle= \langle  P \cap Y_{w,p}\rangle,$$
 contrary to our assumptions. 
\end{proof}

\section{The proofs of Theorem \ref{thm:sol} and Theorem \ref{thm:metanilp}} 

As mentioned in the introduction, now the proof of Theorem \ref{thm:sol} follows easily. 

\begin{proof}[{\textsc{of Theorem \ref{thm:sol}}}]
Recall that $G$ is a finite group in which every nilpotent subgroup generated by $w$-values has rank at most $r$. Let $H$ be a subgroup of $w(G)$ and $P$ a Sylow $p$-subgroup of $H$. By the aforementioned result of Guralnick \cite{gur} and Lucchini \cite{dp} it is sufficient to prove that the rank of $P$ is at most $r$. 

The Ore conjecture that every element of a nonabelian finite simple group is a commutator was famously verified in \cite{lost}. Since $w$ is a multilinear commutator word, we easily deduce that if $G$ is a nonabelian simple group, then every element of $G$ is a $w$-value. So $P$ is a nilpotent subgroup of $G$ generated by $w$-values. By hypotheses, the rank of $P$ is at most $r$. 

Assume now that $G$ is soluble and let $\tilde P$ be a Sylow $p$-subgroup of $w(G)$  containing $P$. By Proposition \ref{prop:focal} $\tilde P$ can be generated by $w$-values. Hence, the rank of $\tilde P$ is at most $r$. So also $P$ has rank at most $r$.  
\end{proof} 

 We now start the preparations for the proof of Theorem  \ref{thm:metanilp}.

The Frattini subgroup of a group $G$ is denoted by $\frat(G)$.   
 Let us denote by $\mathrm{Fit}(G)$ the Fitting subgroup of $G$ and by $F_i(G)$ the $i$th term of the upper Fitting series of $G$, defined recursively by $F_1(G)=\mathrm{Fit}(G)$ and $F_i(G)/F_{i-1}(G)=\mathrm{Fit}(G/F_{i-1}(G))$. If $G$ is a finite soluble group, the least number $h$ with the property that $F_h(G)=G$ is called the Fitting height of $G$. 

The next lemma is quite well known. 

\begin{lem}\label{lem:fitting}
Let  $G$ be a  finite  soluble group  of rank at most $r$. Then the Fitting height of $G$ is at most $2r+1$.  
\end{lem}
\begin{proof}
 Since $\mathrm{Fit}(G)/\frat(G)= \mathrm{Fit}(G/\frat(G) )$
(see e.g. \cite[5.2.15]{Rob}), without loss of generality, we can assume  $\frat(G) = 1$.
 In this case, $F=\mathrm{Fit}(G)$ is a direct product of abelian minimal normal subgroups of $G$, say 
\[F= N_1 \times \ldots \times N_t\]
where each $N_i$ is an elementary abelian $p_i$-group of rank at most $r$.

Set $H_i=G/C_G(N_i)$, for $i=1, \ldots t$. Every $H_i$ is isomorphic to a  soluble linear group acting on $N_i$, where $N_i$ is a vector space of dimension at most $r$. As  a soluble subgroup of $\GL(n, \mathbb{F})$, where $ \mathbb{F}$ is any field, has derived length  at most $2n$ (see for instance \cite[Theorem 6.2A]{dixlin}), the derived length of each $H_i$ is bounded by $2r$. 
Therefore, $G/\cap_{i=1}^t C_G(N_i)$ has derived length at most $2r$. Since $ \cap_{i=1}^t C_G(N_i)=C_G(F)\le F$ 
 (see e.g. \cite[5.4.4]{Rob}) 
 it follows that $G/F$  has derived length at most $2r$. We conclude that $G$ has Fitting height at most $2r+1$. 
\end{proof}

As a corollary of Theorem \ref{thm:sol}  we deduce the following. 
\begin{cor}\label{cor:sol-fit}
Let $w$ be a multilinear commutator word and $K$ a  finite soluble group in which every nilpotent subgroup generated by $w$-values has rank at most $r$. Then the Fitting height of $K$ is bounded in terms of $r$ and $w$.
\end{cor}
\begin{proof} Let $n$  be the height of $w$. As every $\delta_n$-value of $K$ is a $w$-value,
  $K/w(K)$ is soluble of derived length at most $n$. So it is sufficient to bound the Fitting height of $w(K)$. By Theorem \ref{thm:sol} applied to $K$, the verbal subgroup $w(K)$ has rank at most $r+1$. Therefore, by Lemma \ref{lem:fitting}, the Fitting height of $w(K)$ is bounded  in terms of $r$. 
\end{proof}

Every finite group $G$ has a normal series each of whose quotients either is soluble or is a direct product of nonabelian simple groups. The nonsoluble length of $G$, denoted by $\lambda (G)$, was defined in \cite{junta2} as the minimal number of nonsoluble factors in a series of this kind: if
$$
1=G_0\leq G_1\leq \dots \leq G_{2k+1}=G
$$
is a shortest normal series in which  for $i$  even  the quotient $G_{i+1}/G_{i}$ is soluble (possibly trivial), and for $i$ odd the quotient $G_{i+1}/G_{i}$   is a (non-empty) direct product of nonabelian simple groups, then the nonsoluble length $\lambda (G)$ is equal to $k$.

%
%

\begin{prop}\label{la=1} \cite[Proposition 2.2]{DS-lambda}
Let $N,M$ be normal subgroups of $G$ such that $\lambda(G/N)\leq\lambda(G/M)\leq1$. Then $\lambda(G/N\cap M)\leq1$. 
\end{prop}

Given a finite group $G$, we define $T(G)$ as the intersection of all normal subgroups $N$ of $G$ such that $\lambda(G/N)\leq1$. It is easy to deduce from   Proposition \ref{la=1} that $\la(G/T(G))\leq1$ and $\la(G/T(G))=1$ if and only if $G$ is nonsoluble.

It is proved in \cite{junta2} that the nonsoluble length $\lambda(G)$ does not exceed the maximum Fitting height of soluble subgroups of a finite group $G$. A straightforward consequence of this result and Corollary \ref{cor:sol-fit} is the following lemma. 

\begin{lem}\label{cor:non-sol-l}
Let $w$ be a multilinear commutator word and $G$ a  finite group  in which every nilpotent subgroup generated by $w$-values has rank at most $r$. Then the nonsoluble length of $G$ is bounded  in terms of $r$ and $w$. 
\end{lem}

The following well known lemma will be useful.
\begin{lem}\label{lem:H}
Let $N$ be a normal subgroup of a finite group $G$. Then there exists a subgroup $H$ of $G$ such that $G=HN$ and $H \cap N \le \frat (H)$. 
\end{lem}
\begin{proof} 
The lemma  clearly holds if $N \le \frat (G)$, with $H=G$. On the other hand, if $N$ is not a subgroup of $\frat(G)$, then there exists  a proper subgroup of $G$ supplementing $N$. Let $H$ be a subgroup of $G$ which is minimal with respect to the property  that $G=HN$. If $N \cap H$ is not contained in $\frat (H)$, then there exists a proper subgroup $M$ of $H$ such that $H=M(N \cap H)$. Thus $G=NH=MN$, against the minimality of $H$. 
\end{proof}

If $w$ is a multilinear commutator word and $N$ is a normal subgroup of  a group $G$ containing no nontrivial $w$-values, then $N$  centralizes $w(G)$ (see e.g. \cite[Theorem 2.3]{TS} or the comment after Lemma 4.1 in \cite{restricted_centr}). 

The next two lemmas deal with particular cases of Theorem \ref{thm:metanilp}. Clearly, if $G$ is perfect, then $G=w(G)$.

\begin{lem}\label{lem:simple}
Let $w$ be a multilinear commutator word and $G$ a  finite group  in which every metanilpotent subgroup generated by $w$-values has rank at most $r$. 
If $G/\frat(G)$ is a direct product of nonabelian simple groups, then the rank of $G$ is bounded  in terms of $r$ and $w$. 
\end{lem}
\begin{proof} 
Let $P$ be a Sylow $p$-subgroup of $G$ and  set $\Phi=\frat(G)$. As $P\Phi$ is metanilpotent, by assumption  the set $Y= G_w \cap P\Phi$  generates a subgroup of rank at most $r$. It follows from the  Ore conjecture  \cite{lost} that every element of $G/ \Phi$ is a $w$-value. Thus $P$ is contained in the set $Y\Phi$ and so $P\Phi/\Phi\le\langle Y\rangle\Phi/\Phi$ has rank at most $r$. 
 
First assume  that   $G/\Phi$ is a simple group. The subgroup $N$ generated by $G_w \cap \Phi$ is nilpotent, hence, by assumption, its rank is at most $r$. 
 Since the image of $\Phi$ in $G/N$ contains no nontrivial $w$-values of $G/N$, and $w$ is a multilinear commutator, $\Phi/N$ centralizes $w(G/N)=G/N$. So $\Phi/N $ is a quotient of the Schur multiplier of the simple group $G/\Phi$. A corollary of the classification of finite simple groups is that the rank of the Schur multiplier of any such group is at most $2$ (see for example \cite[Table 4.1]{gor-simple}). As $P\Phi/\Phi$, $\Phi/N $ and $N$ have bounded rank, we deduce that $P$ has bounded rank. Since this holds for every prime $p$, the aforementioned result of Guralnick \cite{gur} and Lucchini \cite{dp} implies that the rank of $G$ is bounded. 

Now assume that $G/\Phi=S_1\times \cdots \times S_t$ is  a direct product of  $t >1$ nonabelian simple groups $S_i$. Let $Q$ be a Sylow $2$-subgroup of $G$. By the above argument,  $Q\Phi/\Phi $ has rank at most $r$. Since  each $S_i$ has a nontrivial Sylow $2$-subgroup, we deduce that $t$ is at most $r$. Therefore the lemma follows from the case where $G/\Phi$ is simple. 
\end{proof}

\begin{lem}\label{lem:l=1}
Let $w$ be a multilinear commutator word and $G$ a perfect finite group such that $\lambda(G) =1$. Assume that every metanilpotent subgroup of $G$ generated by $w$-values has rank at most $r$. Then the rank of $G$ is bounded  in terms of $r$ and $w$. 
\end{lem}
\begin{proof} 
As $G$ is perfect  and $\lambda(G) =1$, the quotient group of $G$ over its soluble radical $R(G)$     is a direct product of nonabelian simple groups. Moreover, by Corollary \ref{cor:sol-fit},  the Fitting height of $R(G)$ is bounded in terms of $r$ and $w$.
 So  $G$ has a normal series  of bounded length 
\begin{equation}\label{eq:eq}
 1=G_{0} < G_1 < \cdots < G_{s-1} < G_{s}=G 
 \end{equation} 
 where $G/G_{s-1}$   is a direct product of nonabelian simple groups and each section $G_i/G_{i-1}$ is nilpotent, for $i=1, \dots , s-1$. We argue by induction on the minimal length $s$ of such a series. The case $s=1$ is handled in  Lemma \ref{lem:simple} so we assume that $s>1$.

Let $H$ be a subgroup of $G$ which is minimal with respect to the property that $G=HG_1$ and $H \cap G_1 \le \frat(H)$ (see Lemma \ref{lem:H}). Note that $H/ H \cap G_1$ is perfect since it is isomorphic to $G/G_1$. We have $H=H' (H \cap G_1)$, whence $H'G_1=HG_1=G$. We therefore conclude that $H$ is perfect, by minimality of $H$. Moreover, $\lambda(H)=\lambda(G)=1$. 

Consider the series of $H$ 
 \[  1 \le  G_1 \cap H \le  \cdots  \le G_{s-1}  \cap H \le G_{s}  \cap H= H. \] 
If $s>2$, then $(G_2 \cap H)/(G_1 \cap H )$ is nilpotent. Taking into account that $G_1\cap H\le\frat(H)$, deduce that $G_2 \cap H $ is nilpotent. By induction on the minimal length of a series as in (\ref{eq:eq}), $H$ has bounded rank.  
  
On the other hand, if $s=2$, then $H/\frat(H)$ is a homomorphic image of $G/G_1$, a direct product of nonabelian simple groups, and we can apply Lemma \ref{lem:simple} to conclude that also in this case $H$ has bounded rank.  
  
Now consider the subgroup $N=\langle G_w \cap G_1 \rangle$ generated by the $w$-values of $G$ lying in $G_1$. Since the image of $G_1$ in $G/N$ contains no nontrivial $w$-values of $G/N$, and $w$ is a multilinear commutator, $G_1/N$ centralizes $w(G/N)=G/N$. Set $K=G/N$ and note that $K/Z(K)$ is a homomorphic image of $H$. Therefore $K/Z(K)$ has bounded rank. A theorem of Lubotzky and Mann \cite{powerful} (see also \cite{rank}) now implies that the derived group $K'$ of $K$ has  bounded rank. Since  $G$ is perfect, we conclude that $G/N$ has bounded rank. 
   Finally, note that since $N$ is a nilpotent subgroup generated by $w$-values, it has rank at most $r$ by the hypothesis. Therefore $G$ has bounded rank, as claimed.
\end{proof}

Write $G^{(\infty)}$ for the last term of the derived series of $G.$  Set
$T_1(G)=G^{(\infty)}$   and, by  induction, $T_{i+1}(G)= T(T_{i}(G))$. In view of Proposition \ref{la=1}, it is clear that
 if $T_{i-1}(G) \neq 1$, then  $ T_i(G)$ is the unique smallest normal subgroup $N$ of $G$ such that $\lambda(G/N)=i-1$. Moreover, $\lambda(T_{i}(G)/T_{i+1}(G))=1$ and $T_{i}(G)$ is perfect   for every $i\geq1$ such that $T_i(G) \neq 1$.

\begin{proof}[{\textsc{of Theorem \ref{thm:metanilp}}}]
Recall that $w$ is a multilinear commutator word and $G$ a finite group in which every metanilpotent subgroup generated by $w$-values has rank at most $r$. We want to prove that the rank of the verbal subgroup $w(G)$ is bounded in terms of $r$ and $w$ only. By Corollary \ref{cor:non-sol-l}, the nonsoluble length  $\lambda=\lambda(G)$ of $G$ is bounded  in terms of $r$ and $w$. We argue by induction on $\lambda$. If $\lambda(G) =0$, then $G$ is soluble, and the result follows from Theorem \ref{thm:sol}. Assume that $\lambda \ge1$, and let $N=T_\lambda$. Note that $N$ is perfect. Moreover $\lambda(G/N)=\la-1$ and $\la(N)=1$. By Lemma \ref{lem:H}, there exists a subgroup $H$ of $G$ such that $G=HN$ and $H \cap N \le \frat (H)$. Since  $\lambda(H/H\cap N) \le \lambda(G/N)=\la-1$ and  
 $H \cap N \le \frat (H)$ is soluble, the nonsoluble length of $H$ is at most $\la-1$. 
 As $H$ inherits the assumptions, by induction $w(H)$ has bounded rank. 
  Moreover, as $N$ is perfect and 
 $\la(N)=1$,  we deduce from Lemma \ref{lem:l=1} that $N$ has bounded rank. 
Now $w(G)/w(G)\cap N$ is isomorphic to $w(G)N/N=w(H)N/N$, so it has bounded rank.  As also $N$ has  bounded rank, we conclude that $w(G)$ has bounded rank. 
\end{proof}

\end{document}